\documentclass{article}

\usepackage{amsfonts}
\usepackage{amsthm}
\usepackage{amsmath}
\usepackage{amssymb}
\usepackage{mathtools}
\usepackage{authblk}
\usepackage{xcolor}
\usepackage{soul}
\usepackage{mathrsfs}
\usepackage{tabularx}
\usepackage[top=1in, bottom=1in, left=1in, right=1in]{geometry}
\usepackage{graphicx}
\usepackage{cite}

\providecommand{\U}[1]{\protect\rule{.1in}{.1in}}

\newtheorem{theorem}{Theorem}[section]

\theoremstyle{definition}
\newtheorem{definition}{Definition}[section]
\theoremstyle{remark}
\newtheorem{remark}{Remark}[section]

\numberwithin{equation}{section}
\begin{document}
		
\title{Partially Exact Controllability of Semilinear Heat Exchanger Systems}

\author[1]{Ismail Huseynov}
\author[2]{Arzu Ahmadova }
\author[3]{Agamirza E. Bashirov}

\affil[1]{Department of Mathematical Modeling and Data Analysis, Physikalisch-Technische Bundesanstalt, 10587 Berlin, Germany\\ \texttt{ismail.huseynov@ptb.de}}
\affil[2]{Faculty of Mathematics, University of Duisburg-Essen, 45127, Essen, Germany \\ \texttt{arzu.ahmadova@uni-due.de}}
\affil[3]{Department of Mathematics, Eastern Mediterranean University, Gazimagusa, Mersin 10, 99628, Turkey, \\ \texttt{agamirza.bashirov@emu.edu.tr}}
\maketitle
		
\begin{abstract}
In this paper, we study two semilinear systems describing a monotubular and a two-stream heat exchanger. Neither system is exactly controllable; however, for each we specify a subspace of the state space with respect to which the system is exactly controllable, thus establishing partial exact controllability.
\end{abstract}		

	
\section{Introduction}

Parallel-flow heat exchangers are critical components in thermal systems, facilitating efficient heat transfer between fluids or gases at different temperatures. These devices are integral to a variety of industrial processes and engineering applications, enabling effective thermal energy exchange, as detailed in Mills' monograph \cite{M}. To model and analyze the performance of these systems, several mathematical formulations have been developed that capture the inherent configurations and complexities of the heat exchange process. Notably, the literature has focused on mono-tubular and two-stream parallel-flow heat exchanger equations. These models form the foundation for investigating dynamic behavior, stability, and control properties, particularly in the context of observability, reachability, and controllability -- key concepts in control theory.

In \cite{SH1}, the authors investigate the dynamical properties of parallel-flow heat exchanger equations under the configuration of boundary control at the inlet and observation at the outlet. They show that, a suitably approximated version of the system is observable if both fluid temperatures are measured at the outlet, and it remains observable with respect to the non-negative cone of the state space even if only one temperature measurement is available. Moreover, the system is reachable when both fluid temperatures are controlled at the inlet, and the reachability with respect to the non-negative cone is retained if only one control input is utilized. Further analysis in \cite{SH2} establishes the reachability of the parallel-flow heat exchanger equations. 

Additionally, in \cite{LC}, the authors derive conditions for exact observability of the two-stream parallel-flow heat exchanger equation with boundary inputs and dual temperature measurements, considering a specific class of initial conditions and inputs. As an extension of the two-stream model, three-fluid parallel-channel heat exchangers have been solved analytically in \cite{MC, BM2005} by formulating them as boundary control systems.

It is noteworthy that the mathematical modeling of heat‐exchanger systems often relies on perturbation theory for linear operators on Hilbert spaces. Early works \cite{SH1,SH2,LC,MC,BM2005}  treat the heat‐exchanger equations as abstractly perturbed differential systems.

A key motivation for developing perturbation theory of strongly continuous operator families, especially the left‐ and right‐translation semigroups, stems from their efficacy in formulating evolution equations for various physical processes, including heat exchange. Recent advances have extended these methods for a class of linear operators: for instance, perturbation results for the generators of cosine operator families have been obtained in \cite{cosine1,cosine2}, and applied to a one‐dimensional perturbed wave equation in a fractional setting. Likewise, \cite{nilpotent} demonstrates that nilpotency is preserved under certain semigroup perturbations; this property is then invoked to give a rigorous account of specific thermodynamic behaviors in heat exchangers.

In this paper, our aim is an investigation of controllability properties of semilinear heat exchanger systems. A representative example of these systems is the linear mono-tubular heat exchanger model governed by the partial differential equation:
\begin{equation}
\label{1.1}
\begin{cases}
\frac{\partial T(t,\theta )}{\partial t} = -\frac{\partial T(t,\theta )}{\partial \theta } -aT(t,\theta ) + bu(t,\theta ), \quad t>0, \quad 0 <\theta  < 1, \\
 T(0,\theta )=T_{0}(\theta ),\ T(t,1)=0.
\end{cases}
\end{equation}
Here, $ T(t,\theta ) $ represents the temperature distribution at time $t$ and position $\theta $. The function $T_0(\theta )$ specifies the initial temperature profile, while $T(t,1)=0$ denotes the fixed temperature at the boundary $\theta =1$. In this case, system \eqref{1.1} can be regarded as a perturbation by the physical parameter $a>0$, which represents the thermal capacity of the medium, that is the amount of heat absorbed or released per unit change in temperature.

In \eqref{1.1}, many things depend on the behaviour of the operator $A=-d/d\theta $.  Therefore, letting $x(t)=T(t,\cdot )$ and $u(t)=u(t,\cdot )$ and assuming that $a=0$, we can write system \eqref{1.1} in the abstract form  
\begin{equation}
\label{1.2}
x'(t)=-(d/d\theta )x(t)+bu(t),\quad 0<t\le T,
\end{equation}
where the state space is $X=L_2(0,1)$, $u\in U_{\rm ad}=L_2(0,T;X)$ is a control, the domain of the differential operator $A=-d/d\theta $ is 
\[
D(A)=\{ x\in L_2(0,1): dx/d\theta \in L_2(0,1),\,x(1)=0\} ,
\]
and $b\in \mathbb{R}$ with $b\not= 0$. The operator $A=-d/d\theta $ generates a strongly continuous semigroup 
\begin{equation}
\label{1.3}
[\mathcal{T}(t)x](\theta )=\begin{cases}x(\theta -t) & \text{if } \theta -t\ge 0, \\ 0 & \text{if } \theta -t<0, \end{cases}\quad x\in X,\ 0\le \theta \le 1,\ t\ge 0, 
\end{equation}
which is called a semigroup of \textit{left-translation}. Its adjoint is the semigroup of \textit{right-translation}
\begin{equation}
\label{1.4}
[\mathcal{T}^*(t)x](\theta )=\begin{cases}x(\theta +t) & \text{if } \theta +t\le 1, \\ 0 & \text{if } \theta +t>1, \end{cases}\quad x\in X,\ 0\le \theta \le 1,\ t\ge 0,
\end{equation}
generated by the operator $A^*=d/d\theta $ with the domain 
\[
D(A^*)=\{ x\in L_2(0,1): dx/d\theta \in L_2(0,1),\,x(0)=0\} .
\]
We obtain that 
\[
[\mathcal{T}(s)\mathcal{T}^*(s)x](\theta )=\begin{cases}x(\theta ) & \text{if } s\le \theta \le 1-s,\ 0\le s\le 1/2, \\ 0 & \text{otherwise}. \end{cases}
\]
Then, the respective controllability operator becomes 
\begin{align*}
[Q(t)x](\theta )
    &=b^2\bigg[ \int _0^t\mathcal{T}(s)\mathcal{T}^*(s)x\,ds\bigg] (\theta )\\
    &=b^2\int _{0}^{\min (t,\theta ,1-\theta )}x(\theta )\,ds=b^2\min (t,\theta ,1-\theta )x(\theta ). 
\end{align*}
Therefore, 
\begin{align*}
\langle Q(t)x,x\rangle _X
   &= b^2\int _0^1\min (t,\theta ,1-\theta )x^2(\theta )\,d\theta \\
   &=b^2\int _0^{1/2}\min (t,\theta )x^2(\theta )\,d\theta +b^2\int _{1/2}^1\min (t,1-\theta )x^2(\theta )\,d\theta .
\end{align*}
Applying the second mean value theorem for integrals \cite{WHS}, we obtain the existence $0<\xi <1/2<\eta <1$ such that
\begin{align}
\label{1.5}
\langle Q(t)x,x\rangle _X
    &= b^2\min (t,1/2)\int _\xi ^{1/2}x^2(\theta )\,d\theta +b^2\min (t,1/2)\int _{1/2}^\eta x^2(\theta )\,d\theta \nonumber \\
    &=b^2\min (t,1/2)\int _\xi ^\eta x^2(\theta )\,d\theta ,
\end{align}
which clearly demonstrate that $Q(t)$ is not coercive (even $Q(t)\not> 0$) because $\langle Q(t)x,x\rangle _X=0$ if $x\in X$ is zero on $(\xi ,\eta )$ but nonzero on $[0,\xi ]\cup [\eta ,1]$. Therefore, system \eqref{1.2} fails to be exactly (even approximately) controllable. However, we can talk about steering initial states to points in the subspace
\[
\{ x\in X : x(\theta )=0 \text{ for } \theta \in [0,\xi ]\cup [\eta ,1]\} 
\]
of $X$. This leads us to the concept of \textit{partially exact controllability}.

The ability of a control system to steer any initial state to any target output is called an \textit{output-controllability} which has been extensively developed in \cite{ETLO, Morse, SK, GD, T} for discrete-time systems. In particular, outputs may be projections of the terminal state. This type of output-controllability, which is called as a \text{partial controllability}, has been defined and investigated in \cite{BMSE, BES} for continuous-time linear deterministic and stochastic systems. Later, it was extended to semilinear systems in \cite{BG, BJ}.  Similarly, the concept of partial observability has been developed for semilinear stochastic systems in finite‐dimensional spaces using an analogous framework in \cite{BA}. A sufficient condition for a partially exact controllability of a semilinear control system in the abstract form is proved in \cite{B-AJC} which employs the ideas and method of the proof from \cite{B}. A specific feature of this method is that it avoids fixed point theorems and is based on piecewise construction of steering controls. This method was first proposed in \cite{BG} for proving approximate controllability and developed in the context of approximate controllability in \cite{HL1, HL2, HL6}.  To the context of exact controllability, it was modified in \cite{B} and to the context of process-controllability in \cite{LFL}. 

In this paper, we prove a sufficient condition of partially exact controllability for a semilinear system in the abstract form modifying the result from \cite{B-AJC}. Then we verify this sufficient condition for a semilinear heat exchanger systems and specify subspaces of the state space with regard to which they are partially exact controllable.


\section{Preliminaries}

In this section, we briefly introduce some notations and definitions from functional analysis, perturbation theory, and control theory used throughout this paper.

All norms and scalar products in Banach and Hilbert spaces will be denoted by $\|\cdot\|$ and $\langle \cdot,\cdot \rangle$, respectively. The context will make clear the spaces involved. By $\mathcal{L}(X,U)$, we denote the space of bounded linear operators from the Hilbert space $X$ to the Hilbert space $U$. The shorthand $\mathcal{L}(X,X)=\mathcal{L}(X)$ is frequently used. When $X=\mathbb{R}^n$ and $U=\mathbb{R}^m$, we let $\mathbb{R}^{m\times n}=\mathcal{L}(X,U)$. The adjoint operator is denoted by $A^*$, and for a matrix $A\in \mathbb{R}^{m\times n}$, $A^*$ coincides with its transpose. The identity operator is denoted by $I$, and zero operator, zero vector, and scalar zero are uniformly denoted by $0$. The domain of a linear operator $A$ will be denoted by $D(A)$. If $A\in \mathcal{L}(X,U)$, then $D(A)=X$. 

For a self-adjoint operator $Q\in \mathcal{L}(X)$ (i.e., $Q^*=Q$), we write $Q\ge 0$ if $\langle Qx,x\rangle \ge 0$ for all $x\in X$, and $Q>0$ if $\langle Qx,x\rangle >0$ for all $x\in X$ with $x\not= 0$. $Q\in \mathcal{L}(X)$ is said to be coercive if there exists $c>0$ such that $\langle Qx,x\rangle \ge c\| x\| ^2$ for all $x\in X$. 

We will use the space $L_2(a,b;X)$ of all square integrable (with respect to the Lebesque measure) $X$-valued functions on $[a,b]$. If $X=\mathbb{R}$, then it will be dropped in the notation of this space.

Let $A_0$ be a densely defined closed linear operator on the Hilbert space $X$ generating a strongly continuous semigroup $\{\mathcal{T}(t): t\geq 0\}$ and let $A_1\in \mathcal{L}(X)$. The semigroup generated by $A_1$ has the stronger uniform continuity property and the power series expansion   
\[
e^{A_1t}=\sum _{n=0}^\infty \frac{A_1^nt^n}{n!}.
\]
Therefore, it is reasonable to use the notation $e^{A_1t}$ for this special semigroup. The semigroup $\{ \mathcal{S}(t): t\geq 0\}$ generated by $A_0+A_1$ with $D(A_0+A_1)=D(A_0)$ is called a bounded perturbation of $\{ \mathcal{T}(t): t\geq 0\}$. Below, we will use the following particular case of the theorem from \cite{TR}.
\begin{theorem}
\label{T.2.1}
Let $A_0$ be a densely defined closed linear operator on a Hilbert space $X$ generating a strongly continuous semigroup $\{\mathcal{T}(t):t\geq 0\}$ and let $A_1\in \mathcal{L}(X)$. If $\mathcal{T}(t)$ and $e^{A_1s}$ commute for all positive $t$ and $s$, then the operator $A_0+A_1$ generates the semigroup $\{ \mathcal{T}(t)e^{A_1t}:t\ge 0\} $.
\end{theorem}


\section{Abstract formulation}

In this section, we formulate an abstract result which will be used to establish controllability of heat exchanger systems. Consider the semilinear control system
\begin{equation}
\label{3.1}
 x'(t) = Ax(t)+Bu(t)+f(t,x(t),u(t)),\quad 0<t\le T,
\end{equation}
with $T>0$, where $x$ and $u$ are state and control processes. Assume the following conditions on the entries of \eqref{3.1}:
\begin{itemize}
\item [{\rm (A)}] $H$, $G$, and $U$ are separable Hilbert spaces and $X=H\times G$.
\item [{\rm (B)}] $A$ is a densely defined closed linear operator on $D(A)\subseteq X$, generating a strongly continuous semigroup $\{ \mathcal{S}(t): t\ge 0\} $.
\item [{\rm (C)}] $B$ is a bounded linear operator from $U$ to $X$.
\item [{\rm (D)}] $f$ is a function from $[0,T]\times X\times U$ to $X$, satisfying
       \begin{itemize}
       \item [$\bullet $] $f$ is Lebesgue measurable in $t$ (the first argument);
       \item [$\bullet $] $f$ is Lipschitz continuous in $x$ (the second argument) in the form
       \[
       \| f(t,x,u)-f(t,y,u)\| \le L(t)\| x-y\| ,
       \]
       where $L\in L_1(0,T)$;
       \item [$\bullet $] $f$ is continuous in $u$ (the third argument);
       \item [$\bullet $] $f$ is bounded in the total argument. 
       \end{itemize}      
\end{itemize}

In (A), $X=H\times G$ is a product Hilbert space with the scalar product defined by 
\[
\langle (h,g),(v,w)\rangle _X=\langle h,v\rangle _H+\langle g,w\rangle _G,\quad h,v\in H,\quad g,w\in G.
\]
Here, we regard $X$ as a state space and $U$ as a control space, considering $U_{\rm ad}=L_2(0,T;U)$ as a set of admissible controls. Under conditions (A)--(D), the equation
\begin{equation}
\label{3.2}
x(t)=\mathcal{S}(t)\xi +\int _0^t\mathcal{S}(t-s)(Bu(s)+f(s,x(s),u(s)))\,ds,\quad 0\le t\le T, 
\end{equation}
has a unique continuous solution for every $u\in U_{\rm ad}$ and $x(0)=\xi \in X$ \cite{LY}. The solution $x$ of \eqref{3.2} is called a mild solution of \eqref{3.1}. For preciseness, it will be denoted by $x^{\xi, u}$. Its terminal value $x^{\xi, u}(T)$ ranges in $X$ and can take values in $X\setminus D(A)$. However. the most important values are those which are in $D(A)$. This motivates to modify the concept of exact controllability and partially exact controllability for \eqref{3.1} in the following way. 

Let $L$ be the linear operator from $X$ to $H$ defined by $L(h,g)=h$, $h\in H$, $g\in G$. Then $L^*h=(h,0)$ and $P_L=L^*L$ is a projection operator, projecting $(h,g)\in X=H\times G$ to $(h,0)$. We denote $\tilde{H}=P_L(X)$ and $\tilde{G}=\tilde{H}^\perp $ (orthogonal complement of $\tilde{H}$). Then $X=\tilde{H}\oplus \tilde{G}$ (the direct sum of $\tilde{H}$ and $\tilde{G}$) while $X=H\times G$.
\begin{definition}
\label{D.3.1}
The system in \eqref{3.1} with the set of admissible controls $U_{\rm ad}$ is said to be 
\begin{description}
\item [\rm (a)] {\it Exactly controllable} on $[0,T]$ if for every $\xi , \eta \in X$, there exists $u\in U_{\rm ad}$ such that $x^{\xi ,u}(T)=\eta $.
\item [\rm (b)] {\it $L$-partially exact controllable} on $[0,T]$ if for every $\xi ,\eta \in X$, there exists $u\in U_{\rm ad}$ such that $Lx^{\xi ,u}(T)=L\eta $.
\item [\rm (c)] {\it Exactly controllable to $D(A)$} on $[0,T]$ if for every $\xi \in X$ and $\eta \in D(A)$, there exists $u\in U_{\rm ad}$ such that $x^{\xi ,u}(T)=\eta $.
\item [\rm (d)] {\it $L$-partially exact controllable to $L(D(A))$} on $[0,T]$ if for every $\xi \in X$ and  $\eta \in D(A)$, there exists $u\in U_{\rm ad}$ such that $Lx^{\xi ,u}(T)=L\eta $.
\end{description}
\end{definition}
The $L$-partially exact controllability is motivated from the fact that some systems such as higher order differential equations, delay equations, wave equations, etc.~can be written in the standard form if the dimension of the state space is enlarged. This trend is observed for stochastic systems as well. It is known that colored noise driven stochastic systems can be reduced to white noise driven system if the dimension of the state space is enlarged \cite{BuJ}. In  \cite{B12, B13}, it is shown that wide band noise driven stochastic systems can also be handled in a similar way. According to the discussion given in the Introduction to this paper, the heat exchanger systems fail to be exactly controllable but they are expected to be partially exact controllable. 

For $f(t,x,u)\equiv 0$, system \eqref{3.1} has the linear form:
 \begin{equation}
 \label{3.3}
 y'(t) = Ay(t)+Bu(t),\ 0\le \tau <t\le T.
 \end{equation}
The controllability Gramian $Q(t)$ associated with this linear system is equal to 
 \[
 {Q}(t)=\int _0^t\mathcal{S}(s)BB^*\mathcal{S}^*(s)\,ds,\quad 0<t\le T.
 \]
Define also the $L$-partial controllability Gramian by 
\[
Q_L(t)=LQ(t)L^*.
\]
Here, $\lim _{t\to 0^+}Q_L(t)=Q_L(0)=0$. We will adjust the rate of convergence in this limit assuming the following condition:
\begin{itemize}
\item [\rm (E)] $Q_L(t)$ is coercive for every $0<t\le T$ and $t\| Q^{-1}_L(t)\| $ is bounded on $(0,T]$.
\end{itemize}

The following theorems express two results about $L$-partial exact controllability.
\begin{theorem} {\rm \cite{BMSE}}
\label{T.3.1}  
Under conditions {\rm (A)--(C)}, linear system \eqref{3.3} is $L$-partially exact controllable on the interval $[\tau ,T]$ if and only if $Q_L(T-\tau )$ is coercive. Moreover, for any given $\xi ,\eta \in X$, a control steering $y(\tau )=\xi $ to $y(T)$ with $Ly(T)=L\eta $ can be defined by 
\begin{equation}
\label{3.4}
u(t)=B^*e^{A^*(T-t)}L^*Q_L(T-\tau )^{-1}L\big( \eta -\mathcal{S}((T-\tau )\xi \big) ,\ \tau \le t\le T.
\end{equation} 
\end{theorem}
\begin{theorem} {\rm \cite{B-AJC}}
\label{T.3.2} 
Under conditions {\rm (A)--(E)} and 
\begin{itemize}
\item [{\rm (F$'$)}] $\mathcal{S}(t)\big( \tilde{G}\big) \subseteq \tilde{G}$, $t\ge 0$,
\end{itemize}
semilinear system \eqref{3.1} is $L$-partially exact controllable to $L(D(A))$ on the interval $[0,T]$.
\end{theorem}

Condition (F$'$) is suitable for semigroups generated by the second order differential operators but unsuitable for semigroups driving heat exchanger systems. Therefore, we will replace it with the following alternative condition and run the proof of Theorem \ref{T.3.2} from \cite{B-AJC} under renewed condition:
\begin{itemize}
\item [{\rm (F)}] For sufficiently small values of $\tau $, $P_LS(\tau )x=P_Lx$ for all $x\in X$.
\end{itemize} 
\begin{remark}
\label{R.3.1}
Condition (F) may seen impossible. However, we are going to verify it for the left translation semigroup $\mathcal{T}$ defined by \eqref{1.3}. Define the subspace of $X=L_2(0,1)$ as
\[
\tilde{H}=\{ x\in X: x(\theta )=0 \text{ if }\theta \in [0,\varepsilon ]\cup [1-\varepsilon ,1]\} ,
\] 
where $0<\varepsilon <1-\varepsilon <1$. Then for $0<\tau <\varepsilon $, we have 
\[
 [\mathcal{T}(\tau )x](\theta )=\begin{cases}x(\theta -\tau ) & \text{if } \tau \le \theta \le 1 \\ 0 & \text{otherwise} \end{cases}=\begin{cases}x(\theta ) & \text{if } 0\le \theta \le 1-\tau ,\\ 0 & \text{otherwise}, \end{cases}  
\]
implying 
\[
[P_L\mathcal{T}(\tau )x](\theta )=\begin{cases}x(\theta ) & \text{if } \varepsilon \le \theta \le 1-\varepsilon \\ 0 & \text{otherwise} \end{cases} =[P_Lx](\theta ).
\]
Similarly, if $0<\tau <\varepsilon $, then $P_L\mathcal{T}^*(\tau )x=P_Lx$.
\end{remark}
\begin{theorem}
\label{T.3.3}
Under conditions {\rm (A)--(F)}, system \eqref{3.1} is $L$-partially exact controllable to $L(D(A))$ on the interval $[0,T]$.
\end{theorem}
\begin{proof}
Let the initial state $x_0\in X$ and $\eta \in D(A)$ be arbitrarily fixed. We construct a control $u$ steering $x_0$ to $\eta $ in the following way. 

Denote $\tau _n=T/2^n$ for $n=1,2,\ldots $ Obviously, $\sum _{n=1}^\infty \tau _n=T$. Let 
 \[
 t_0=0,\ t_1=\tau _1, \ldots ,\ t_n=\sum _{k=1}^n\tau _k,\ldots 
 \]
Then 
\[
\lim _{n\to \infty }t_n=\sum _{k=1}^\infty \tau _k=T.
\]
By Theorem \ref{T.3.1}, linear system \eqref{3.3} is $L$ partially exact controllable on  $[t_0,t_1]$ and the control
 \[
 u_1(t)=B^*\mathcal{S}^*(t_1-t)L^*Q^{-1}_L(\tau _1)L\mathcal{S}(\tau _1)(\eta -x_0),\ t_0\le t\le t_1,
 \]
steers $x_0$ to $y(t_1)$ with $Ly(t_1)=L\mathcal{S}(\tau _1)\eta $, that is,
 \[
 L\mathcal{S}(\tau _1)\eta =L\mathcal{S}(\tau _1)x_0+\int _{t_0}^{t_1}L\mathcal{S}(t_1-s)Bu_1(s)\,ds.
 \]
Letting $u(t)=u_1(t)$ for $t_0\le t\le t_1$, we obtain 
 \[
L x^{x_0,u}(t_1)=L\mathcal{S}(\tau _1)\eta +\int _{t_0}^{t_1}L\mathcal{S}(t_1-s)f\big( s,x^{x_0 ,u}(s),u(s)\big) \,ds.
 \] 
For brevity, we denote $x^{x_0 ,u}(t_1)=x_1$.
 
Considering \eqref{3.1} on the interval $[t_1,t_2]$, we let  
 \[
 u_2(t)=B^*\mathcal{S}^*(t_2-t)L^*Q^{-1}_L(\tau _2)L\mathcal{S}(\tau _2)(\eta -x_1),\ t_1\le t\le t_2.
 \]
By Theorem \ref{T.3.1}, the control $u_2$ steers $x_1$ to $y(t_2)$ with $Ly(t_2)=L\mathcal{S}(\tau _2)\eta $, that is,    
 \[
 L\mathcal{S}(\tau _2)\eta =L\mathcal{S}(\tau _2)x_1+\int _{t_1}^{t_2}L\mathcal{S}(t_2-s)Bu_2(s)\,ds.
 \]
Letting $u(t)=u_2(t)$ for $(t_1,t_2]$, we obtain
\[
Lx^{x_0,u}(t_2) =L\mathcal{S}(\tau _2)\eta +\int _{t_1}^{t_2}L\mathcal{S}(t_2-s)f\big( s,x^{x_0,u}(s),u(s)\big) \,ds.
\]
For brevity, let $x^{x_0 ,u}(t_2)=x_2$. Continuing this procedure, we recursively obtain the sequence of controls  
 \begin{equation}
 \label{3.5}
 u_n(t)=B^*\mathcal{S}^*(t_n-t)L^*Q^{-1}_L(\tau _n)L\mathcal{S}(\tau _n)(\eta -x_{n-1}),\ t_{n-1}\le t\le t_n.
 \end{equation}
Using these controls as parts of the total control, we define  
 \[
 u(t)=\left\{ \begin{array}{ll} u_1(t), & \text{if}\ t_0\le t\le t_1,\\
                                           u_2(t), & \text{if}\ t_1< t\le t_2,\\
                                    \ldots\ldots & \ldots \ldots \ldots \ldots \\
                                           u_n(t), & \text{if}\ t_{n-1}< t\le t_n,\\
                                   \ldots \ldots & \ldots \ldots \ldots \ldots 
 \end{array}\right.                                               
 \]
Then
 \begin{equation}
 \label{3.6}
 L\mathcal{S}(\tau _n)\eta =L\mathcal{S}(\tau _n)x_{n-1}+\int _{t_{n-1}}^{t_n}L\mathcal{S}(t_n-s)Bu(s)\,ds,
 \end{equation}
implying 
\begin{equation}
\label{3.7}
Lx^{x_0 ,u}(t_n)=L\mathcal{S}(\tau _n)\eta +\int _{t_{n-1}}^{t_n}L\mathcal{S}(t_n-s)f\big( s,x^{x_0 ,u}(s),u(s)\big) \,ds. 
 \end{equation} 
Denote $x^{x_0 ,u}(t_n)=x_n$. To estimate $\| Lx_n-L\eta \| $, let 
\[
K=\sup _{[0,T]}\| \mathcal{S}(t)\| \ \text{and}\ M=\sup _{[0,T]\times X\times U}\| f(t,x,u)\| .
\]
Using $\| L\| \le 1$ and \eqref{3.7}, we have 
\begin{align}
\label{3.8}
\| Lx_n-L\eta \| 
    & \le \| L( x_n-\mathcal{S}(\tau _n)\eta )\| +\| L(\mathcal{S}(\tau _n)\eta -\eta )\|  \nonumber \\
    & \le \| L(\mathcal{S}(\tau _n)\eta -\eta )\|  +\int _{t_{n-1}}^{t_n} \big\| L\mathcal{S}(t_n-s)f\big( s,x^{x_0 ,u}(s),u(s)\big) \big\| \,ds\nonumber \\
    & \le \| \mathcal{S}(\tau _n)\eta -\eta \|  +\int _{t_{n-1}}^{t_n}\| \mathcal{S}(t_n-s)\| \big\| f\big( s,x^{x_0 ,u}(s),u(s)\big) \big\| \,ds\nonumber \\
    & \le \| \mathcal{S}(\tau _n)\eta -\eta \|  +MK \tau _n,\ n=1,2,\ldots 
 \end{align}
Therefore, $\lim _{n\to \infty }Lx_n=L\eta $ since $\mathcal{S}(t)$ is strongly continuous and $\lim _{n\to \infty }\tau _n=0$. 

Now let us prove that $u\in U_{\rm ad}$. By construction, each piece $u_n$ of $u$ is continuous on the interval $(t_{n-1},t_n]$, implying the measurability of $u$. Moreover,
\begin{align}
\label{3.9}
\int _{t_{n-1}}^{t_n}\| u_n(t)\| ^2\,dt 
 & = \int _{t_{n-1}}^{t_n}\big\| B^*\mathcal{S}^*(t_n-t)L^*Q^{-1}_L(\tau _n)L \mathcal{S}(\tau _n)(\eta -x_{n-1})\big\| ^2 \,dt\nonumber \\
 & =\int _{t_{n-1}}^{t_n}\big\langle  L\mathcal{S}(t_n-t)BB^*\mathcal{S}^*(t_n-t)L^*Q^{-1}_L(\tau    _n)\alpha _n,Q^{-1}_L(\tau _n)\alpha _n\big\rangle \,dt\nonumber \\
 & =\big\langle \alpha _n,Q^{-1}_L(\tau _n)\alpha _n \big\rangle \le \big\| Q^{-1}_L(\tau _n)\big\| \| \alpha _n\| ^2,\ n=1,2,\ldots , 
\end{align}
where $\alpha _n=L\mathcal{S}(\tau _n)(\eta -x_{n-1})$. Since $\lim _{n\to \infty } \tau _n=0$, by condition (F), we can use the equality 
\[
P_L\mathcal{S}(\tau _n)(\eta -x_{n-1})=P_L(\eta -x_{n-1})
\]
for sufficiently large $n$. Then, by \eqref{3.8},
\begin{align*}
\| \alpha _n\| ^2
   & =\langle L\mathcal{S}(\tau _n)(\eta -x_{n-1}),L\mathcal{S}(\tau _n)(\eta -x_{n-1})\rangle =\langle L^*L\mathcal{S}(\tau _n)(\eta -x_{n-1}),\mathcal{S}(\tau _n)(\eta -x_{n-1})\rangle \\
   & =\langle P_L\mathcal{S}(\tau _n)(\eta -x_{n-1}),\mathcal{S}(\tau _n)(\eta -x_{n-1})\rangle =\langle P_L^2\mathcal{S}(\tau _n)(\eta -x_{n-1}),\mathcal{S}(\tau _n)(\eta -x_{n-1})\rangle \\
   & =\langle P_L\mathcal{S}(\tau _n)(\eta -x_{n-1}),P_L\mathcal{S}(\tau _n)(\eta -x_{n-1})\rangle =\| P_L\mathcal{S}(\tau _n)(\eta -x_{n-1})\| ^2\\
   & =\| P_L(\eta -x_{n-1})\| ^2=\| L^*L(\eta -x_{n-1})\| ^2\le \| L(\eta -x_{n-1})\| ^2\\
   & \le (\| \mathcal{S}(\tau _{n-1})\eta -\eta \|  +MK \tau _{n-1})^2,\ n=2,3,\ldots 
\end{align*}
By \eqref{3.9}, this implies
\begin{align*}
\int _{t_{n-1}}^{t_n}\| u_n(t)\| ^2\,dt
  &\le \big\| Q^{-1}_L(\tau _n)\big\| (\| \mathcal{S}(\tau _{n-1})\eta -\eta \|  +MK \tau _{n-1})^2\\
  &=\tau _n^2\| Q^{-1}_L(\tau _n)\| \cdot \frac{\tau _{n-1}^2}{\tau _n^2}\cdot \bigg( \bigg\| \frac{\mathcal{S}(\tau _{n-1})\eta -\eta }{\tau _{n-1}}\bigg\| +MK\bigg) ^2,\\ 
  &=4\tau _n^2\| Q^{-1}_L(\tau _n)\| \bigg( \bigg\| \frac{\mathcal{S}(\tau _{n-1})\eta -\eta }{\tau _{n-1}}\bigg\| +MK\bigg) ^2,\ n=2,3,\ldots   
\end{align*}
Here, $\tau _n\| Q^{-1}_L(\tau _n)\| $ is bounded by condition (E). Also, $\| (\mathcal{S}(\tau _{n-1})\eta -\eta )/\tau _{n-1}\| $ is bounded since $\eta \in D(A)$. Therefore, there is a constant $c>0$ such that
\[
\int _{t_{n-1}}^{t_n}\| u_n(t)\| ^2\,dt\le c\tau _n,\ \ n=2,3,\ldots , 
\]
implying
\begin{align*}
\int _0^T\| u(t)\| ^2\, dt
   & =\int _0^{t_1}\| u(t)\| ^2\, dt+\int _{t_1}^T\| u(t)\| ^2\, dt=\int _0^{t_1}\| u(t)\| ^2\, dt+\sum _{n=1}^\infty \int _{t_n}^{t_{n+1}}\| u_{n+1}(t)\| ^2\,dt\\
   &  \le \int _0^{t_1}\| u(t)\| ^2\, dt+c\sum _{n=1}^\infty \tau_{n+1}=\int _0^{t_1}\| u(t)\| ^2\, dt+cT\sum _{n=1}^\infty \frac{1}{2^{n+1}}\\
   & =\int _0^{t_1}\| u(t)\| ^2\, dt+\frac{cT}{2}<\infty . 
\end{align*}
This completes the proof.
\end{proof}


\section{Monotubular heat exchanger equation}
 
We consider a controlled semilinear system described by the following monotubular heat exchanger equation
\begin{equation}
\label{4.1}
\begin{cases}
\frac{\partial x(t,\theta )}{\partial t} = -\frac{\partial x(t,\theta )}{\partial \theta } -ax(t,\theta ) + bu(t,\theta )+f(t,x(t,\theta ),u(t,\theta )), \quad 0<t\le T, \quad 0 <\theta < 1, \\
 x(0,\theta )=x_{0}(\theta ),\ x(t,1)=0,\quad 0<t\le T.
\end{cases}
\end{equation}
Let $X= L_2(0, 1)$ be the state space. Define the operators $A_0$, $A_1$, and $B$ as follows:
\begin{itemize}
\item
$A_0x=-dx/d\theta $, $x\in D(A_0) =\{ x\in X : dx/d\theta \in X,\, x(1)=0\} $,
\item
$[A_1x](\theta)=-ax(\theta )$, $0\le \theta \le 1$, $x\in X$,
\item
$[Bu](\theta )=bu(\theta )$, $0\le \theta \le 1$, $u\in X$,
\end{itemize}
where $b\not= 0$ and $a>0$ are constants. Here, $A_0$ is the infinitesimal generator of the left-translation semigroup $\{ \mathcal{T}(t): t\ge 0\} $ from \eqref{1.3}. Then, by Theorem \ref{T.2.1}, the semigroup generated by $A=A_0+A_1$ associated with the mono-tubular heat exchanger equation is
\begin{equation}
[\mathcal{S}(t)x](\theta ) = 
\begin{cases}
e^{-at}x (\theta -t), & \text{if } \theta -t\ge 0, \\
0, & \text{if } \theta -t<0,
\end{cases}
\end{equation}
for $t \ge 0$ and $0\le \theta \le 1$.
Since $A_1^*=A_1$, it follows that 
\[
A^*=A_0^*+A_1^*=(-\partial /\partial \theta )^*+A_1= (\partial /\partial \theta )+A_1.
\]
The semigroup generated by $A_0^*$ is the right-translation semigroup $\{ \mathcal{T}^*(t): t\ge 0\} $ from \eqref{1.4} and, therefore,
$A^*$ generates the semigroup
\begin{equation}
[\mathcal{S}^{*}(t)x](\theta ) = 
\begin{cases}
e^{-at}x(\theta +t), & \text{if } \theta +t\le 1, \\
0, & \text{if } \theta +t>1.
\end{cases}
\end{equation}
for $t\ge 0$ and $0\le \theta \le 1$. Then the controllability Gramian $Q(t)$ becomes
\begin{align*}
[Q(t)x](\theta ) 
    &= \int_0^t\mathcal{S}(s)BB^*\mathcal{S}^*(s)x(\theta )\,ds\\
    &= b^2\int_0^{\min (t,\theta ,1-\theta )}e^{-2as}x(\theta )\,ds=\frac{b^2\big( 1-e^{-2a\min(t,\theta ,1-\theta )}\big) }{2a}x(\theta ).
\end{align*}
Therefore,
\begin{equation}
\label{4.4}
\langle Q(t)x,x\rangle _X=\frac{b^2}{2a}\int _0^1\big( 1-e^{-2a\min (t,\theta ,1-\theta )}\big) x^2(\theta )\,d\theta .
\end{equation}

Now, take any $0<\varepsilon <1/2$ and consider $X_\varepsilon =L_2(\varepsilon ,1-\varepsilon )$. Respectively, let
\[
\tilde{X}_\varepsilon =\{ x\in L_2(0,1): x(\theta )=0\text{ if }\theta \in [0,\varepsilon )\cup (1-\varepsilon ,1]\}.
\]
Define $L_\varepsilon $ as an operator assigning to $x\in L_2(0,1)$ its restriction to $[\varepsilon ,1-\varepsilon ]$. Then $P_{L_\varepsilon }=L^*_\varepsilon L_\varepsilon $ is the projection operator from $L_2(0,1)$ to its subspace $\tilde{X}_\varepsilon $. From \eqref{4.4}, we have
\[
\langle Q_{L_\varepsilon }(t)x,x\rangle _{X_\varepsilon }=\langle L_\varepsilon Q(t)L_\varepsilon ^*x,x\rangle _{X_\varepsilon }=\frac{b^2}{2a}\int _\varepsilon ^{1-\varepsilon }\big( 1-e^{-2a\min (t,\theta ,1-\theta )}\big) x^2(\theta )\,d\theta .
\]
For $0<t<\varepsilon $, this becomes
\[
\langle Q_{L_\varepsilon }(t)x,x\rangle _{X_\varepsilon }=\frac{b^2(1-e^{-2at})}{2a}\int _\varepsilon ^{1-\varepsilon }x^2(\theta )\,d\theta ,
\]
implying that $Q_{L_\varepsilon }(t)$ is coercive. Therefore, $Q_{L_\varepsilon }(t)$ is coercive for $\varepsilon \le t\le T$ as well because $Q_{L_\varepsilon }(t)\ge Q_{L_\varepsilon }(s)$ for $t\ge s$. Finally, using the inequality $e^z-1\ge z$, we obtain 
\[
\frac{b^2(1-e^{-2at})}{2a}\ge b^2t,
\] 
demonstrating $t^{-1}\langle Q_{L_\varepsilon }(t)x,x\rangle _{X_\varepsilon }\ge b^2\| x\| _{X_\varepsilon }^2$ and implying $t\| Q^{-1}_{L_\varepsilon }\| _{X_\varepsilon }\le b^{-2}$. Therefore, condition (E) holds. Condition (F) holds as well because of Remark \ref{R.3.1}. Thus, assuming that $f$ satisfies condition (D), we obtain that system \eqref{4.1} is $L_\varepsilon $-partially exact controllable for all $0<\varepsilon <1/2$. In other words, {\it for all $0<\varepsilon <1/2$ and for every absolutely continuous function $y$ with $L_2$-derivative which has zero values on the intervals $[0,\varepsilon )$ and $(1-\varepsilon ,1]$, any initial function from $L_2(0,1)$ can be steered to $x_T$ with $x_T(\theta )=y(\theta )$ on $[\varepsilon ,1-\varepsilon]$}.

    
\section{Two-stream parallel-flow heat exchanger equation} 

Monotubular heat exchanger equation from the previous section motivates to consider two-stream parallel-flow heat exchanger equation as well.  Now, we consider a semilinear system described by the following two-stream parallel-flow heat exchanger equation
\begin{equation}
\label{5.1}
\begin{cases}
     \frac{\partial x_1(t,\theta)}{\partial t} = -\frac{\partial x_1(t,\theta)}{\partial \theta} +h_1(x_2(t,\theta) - x_1(t,\theta)) + b_1u_1(t,\theta)\\
     \ \ \ \ \ \ \ \ \ \ \ \ \ +f_1(x_1(t,\theta ),x_2(t,\theta ),u_1(t,\theta ),u_2(t,\theta )), \quad 0<t\le T, \quad 0 <\theta< 1, \\
     \frac{\partial x_{2}(t,\theta)}{\partial t} = -\frac{\partial x_{2}(t,\theta)}{\partial \theta} +h_{2}(x_{1}(t,\theta) - x_{2}(t,\theta))+ b_{2}u_{2}(t,\theta)\\
      \ \ \ \ \ \ \ \ \ \ \ \ \ +f_2(x_1(t,\theta ),x_2(t,\theta ),u_1(t,\theta ),u_2(t,\theta )), \quad 0<t\le T, \quad 0 <\theta< 1, \\
     x_1(0,\theta) = x_{10}(\theta), \quad x_2(0,\theta) = x_{20}(\theta), \quad 0 \leq \theta \leq 1, \\
     x_1(t, 1) = x_2(t, 1)=0, \quad 0<t\le T,
\end{cases}
\end{equation}
where $x_{1}(t,\theta)$ and $x_{2}(t,\theta)$ represent the cold and hot stream temperatures at time $t$ and position $\theta$, $x_{10}$ and $x_{20}$ the initial temperature profiles, $u_{1}(t,\theta)$ and $u_{2}(t,\theta)$ are control functions, $b_1$ and $b_2$ are nonzero constants, $h_{1}$ and $h_{2}$ are positive physical constants related with the \textit{flow thermal capacity}, which reflects the capacity of gaining or losing heat for a unit temperature change of the inner and outer tubes, respectively.

Denote $X= L_2(0, 1)$. Let $X^2=L_2(0, 1)\times  L_2(0, 1)$ be the state space. Define the operators $A_0$, $A_1$, and $B$ as follows:
\begin{itemize}
\item
$A_0 \begin{bmatrix} x_1\\x_2\end{bmatrix}=\begin{bmatrix}
 -dx_1/d\theta & 0 \\
0 & -dx_2/d\theta 
\end{bmatrix}$,\ $\begin{bmatrix} x_1\\x_2\end{bmatrix}\in D(A_0)$,
\item
$A_1  \begin{bmatrix} x_1 \\ x_2 \end{bmatrix} = \begin{bmatrix} -h_1 & h_1 \\ h_2 & -h_2  \end{bmatrix}  \begin{bmatrix} x_{1} \\ x_{2} \end{bmatrix} = \begin{bmatrix} h_1(x_2-x_1) \\ h_2(x_1-x_2) \end{bmatrix}$, $x_1, x_2\in X$, 
\item
$Bu=\begin{bmatrix} b_1 & 0 \\ 0 & b_2 \end{bmatrix}
\begin{bmatrix} u_1 \\ u_2 \end{bmatrix}$, $u_1,u_2\in X$,
\end{itemize}
where 
\[
D(A_0)=\bigg\{\begin{bmatrix} x_1\\x_2\end{bmatrix}\in X^2: dx_1/d\theta , dx_2/d\theta \in X, x_1(1)=x_2(1)=0\bigg\} .
\]
Here, $A_0$ is the infinitesimal generator of the semigroup 
\begin{equation*}
\mathcal{U}(t) = \begin{bmatrix} \mathcal{T}(t) & 0 \\ 0 & \mathcal{T}(t) \end{bmatrix}=\mathcal{T}(t) \begin{bmatrix} I & 0 \\ 0 & I \end{bmatrix},\ t\ge 0,
\end{equation*}
where $\{ \mathcal{T}(t): t\ge 0\} $ denotes the left-translation semigroup defined by \eqref{1.3}. Also, using the diagonalization technique of the matrix exponential function, it can be computed that
\begin{equation}
 e^{A_1t}= \begin{bmatrix}
 \frac{h_2}{h_1+h_2} + \frac{h_1}{h_1+h_2}e^{-(h_1+h_2)t} & \frac{h_1}{h_1+h_2}(1-e^{-(h_1+h_2)t})\\
\frac{h_2}{h_1+h_2}(1-e^{-(h_1+h_2)t}) & \frac{h_1}{h_1+h_2} + \frac{h_2}{h_1+h_2}e^{-(h_1+h_2)t}
\end{bmatrix}.
\end{equation}
Since $\mathcal{U}(t)$ and $e^{A_1s}$ commute, by Theorem \ref{T.2.1}, we obtain that $A=A_0+A_1$ generates the semigroup $\mathcal{S}(t)=\mathcal{U}(t)e^{A_1t}=e^{A_1t}\mathcal{U}(t)$. Then $A^*=A^*_0+A^*_1$ generates the semigroup $\mathcal{S}^*(t) = \mathcal{U}^*(t)e^{A^*_1t}=e^{A^*_1t}\mathcal{U}^*(t)$, where
\[
\mathcal{U}^*(t) = \begin{bmatrix} \mathcal{T}^*(t) & 0 \\ 0 & \mathcal{T}^*(t) \end{bmatrix}= \mathcal{T}^*(t)\begin{bmatrix} I & 0 \\ 0 & I \end{bmatrix},
\]
$\{ \mathcal{T}^*(t):t\ge 0\} $ is a semigroup of right-translation from \eqref{1.4}, and 
\[
 e^{A^*_1t} = \begin{bmatrix}
 \frac{h_2}{h_1+h_2} + \frac{h_1}{h_1+h_2}e^{-(h_1+h_2)t} & \frac{h_2}{h_1+h_2}(1-e^{-(h_1+h_2)t})\\
\frac{h_1}{h_1+h_2}(1-e^{-(h_1+h_2)t}) & \frac{h_1}{h_1+h_2} + \frac{h_2}{h_1+h_2}e^{-(h_1+h_2)t}.\end{bmatrix}
\]

To simplify the complexity of further calculations, we will consider a particular case when $h_1=h_2=1/2$ and $b_1=b_2=1$ noting that the conclusion of this section remains valid for general $b_1\not= 0$, $b_2\not= 0$, $h_1>0$, and $h_2>0$. Then for $0\le \theta \le 1$, we have
\begin{align*}
[Q(t)x](\theta ) 
    &= \int_0^t[\mathcal{S}(s)BB^*\mathcal{S}^*(s)x](\theta )\,ds=\int_0^t[\mathcal{T}(s)e^{A_1s} 
       e^{A^*_1s}\mathcal{T}^{*}(s) x](\theta)\,ds\\
    &= \frac{1}{2}\int_0^{\min (t,\theta ,1-\theta )}\begin{bmatrix}
       1+e^{-s} & 1-e^{-s} \\ 1-e^{-s} & 1+e^{-s} 
       \end{bmatrix} ^2
       \begin{bmatrix} x_{1}(\theta) \\ x_{2}(\theta) \end{bmatrix}\, ds\\  
     &= \int_0^{\min (t,\theta ,1-\theta )}\begin{bmatrix}
       1+e^{-2s} & 1-e^{-2s} \\ 1-e^{-2s} & 1+e^{-2s}
       \end{bmatrix} 
       \begin{bmatrix} x_{1}(\theta) \\ x_{2}(\theta \end{bmatrix}\, ds\\
    &=\frac{1}{2}\begin{bmatrix}
      u(t,\theta) & v(t,\theta) \\ v(t,\theta) & u(t,\theta) \end{bmatrix}
     \begin{bmatrix} x_1(\theta) \\ x_2(\theta) \end{bmatrix},
\end{align*}
where
\begin{align*}
u(t,\theta)&=2\min ( t,\theta , 1-\theta )-e^{-2\min ( t,\theta , 1-\theta )}+1, \\
v(t,\theta)&=2\min ( t,\theta , 1-\theta )+e^{-2\min ( t,\theta , 1-\theta )}-1. 
\end{align*}
Therefore,
\begin{equation} 
\label{5.3}
\langle Q(t)x,x\rangle_{X^2} = \frac{1}{2}\int_0^1\big( u(t,\theta )\big( x_1^2(\theta )+x_2^2(\theta )\big) +2v(t,\theta ) x_1(\theta )x_2(\theta )\big) \,d\theta 
\end{equation}
Similar to derivation of \eqref{1.5}, by the second mean value theorem for integrals, we can find $0<\xi _1<1/2<\eta _1<1$ and $0<\xi _2<1/2<\eta _2<1$ such that 
\begin{align*}
2\langle Q(t)x,x\rangle_{X^2}
    &= (2\min (t,1/2)-e^{2\min (t,1/2)}+1)\int _{\xi _1}^{\eta _1}(x_1^2(\theta )+x_2^2(\theta ))\,d\theta \\
    &\ \ \ +(2\min (t,1/2)+e^{2\min (t,1/2)}-1)\int _{\xi _2}^{\eta _2}2x_1(\theta )x_2(\theta )\,d\theta .
\end{align*}
We obtain that if $x_1(\theta )=x_2(\theta )=0$ for $0<\min (\xi_1,\xi_2)\le \theta \le \max (\eta _1,\eta _2)<1$, then $\langle Q(t)x,x\rangle_{X^2}=0$. Therefore, $Q(t)$ fails to be coercive (even $Q(t)\not> 0$).

Now, take any $0<\varepsilon <1/2$ and consider $X^2_\varepsilon =L_2(\varepsilon ,1-\varepsilon )\times L_2(\varepsilon ,1-\varepsilon )$. Respectively, let
\[
\tilde{X}^2_\varepsilon =\{ x\in X^2: x(\theta )=0\text{ if }\theta \in [0,\varepsilon )\cup (1-\varepsilon ,1]\}.
\]
Define $L_\varepsilon $ as an operator assigning to $x\in X^2$ its restriction to $[\varepsilon ,1-\varepsilon ]$. Then $P_{L_\varepsilon }=L^*_\varepsilon L_\varepsilon $ is the projection operator from $X^2$ to its subspace $\tilde{X}^2_\varepsilon $. Similar to \eqref{5.3}, we have 
\begin{equation} 
\label{5.4}
\langle Q_{L_\varepsilon }(t)x,x\rangle_{X^2_\varepsilon } = \frac{1}{2}\int_\varepsilon ^{1-\varepsilon }\big( u(t,\theta )\big( x_1^2(\theta )+x_2^2(\theta )\big) +2v(t,\theta ) x_1(\theta )x_2(\theta )\big) \,d\theta 
\end{equation}
For $0<t<\varepsilon $, $u$ and $v$ become 
\[
u(t,\theta)=2t-e^{-2t}+1 \text{ and }v(t,\theta)=2t+e^{-2t}-1. 
\]
Then from \eqref{5.4},
\[
\langle Q_{L_\varepsilon }(t)x,x\rangle_{X^2_\varepsilon } = \int_\varepsilon ^{1-\varepsilon }\bigg( t(x_1(\theta )+x_2(\theta ))^2+\frac{t(e^{-2t}-1)}{-2t}(x_1(\theta )-x_2(\theta ))^2\bigg) \,d\theta
\]
Using the inequality $e^z-1\ge z$, we arrive to
\begin{align}
\label{5.5}
\langle Q_{L_\varepsilon }(t)x,x\rangle_{X^2_\varepsilon } 
   & \ge  t\int_\varepsilon ^{1-\varepsilon }\big((x_1(\theta )+x_2(\theta ))^2+(x_1(\theta )-x_2(\theta ))^2\big) \,d\theta \nonumber \\
   & \ge t\int_\varepsilon ^{1-\varepsilon }\big(x_1(\theta )^2+x_2(\theta )^2)\,d\theta =t\| x\| ^2_{X^2_\varepsilon }.
\end{align}
Thus, $Q_{L_\varepsilon }(t)$ is coercive for all $0<t<\varepsilon $. Since $Q_{L_\varepsilon }(t)\le Q_{L_\varepsilon }(s)$ for $0\le t\le s$, $Q_{L_\varepsilon }(t)$ becomes coercive for all $0<t\le T$. Also, \eqref{5.5} implies that $t\| Q^{-1}_{L_\varepsilon }(t)\| $ is bounded on $[0,\varepsilon ]$. Therefore, it is bounded on $[0,T]$ because it is continuous on $[\varepsilon ,T]$. So, condition (E) holds. Condition (F) also holds because of Remark \ref{R.3.1}. Thus, assuming that $f_1$ and $f_2$ form a function $f$ on $[0,T]\times X^2\times X^2$, we obtain that the system \eqref{5.1} is $L_\varepsilon $-partially exact controllable for all $0<\varepsilon <1/2$.


\section{Conclusion}

In this paper, a sufficient condition for a partially exact controllability to the domain of an abstract semilinear control system is proved. This sufficient condition is tested on semilinear monotubular and two stream parallel-flow heat exchanger systems. It is shown that they are not exactly controllable, but partially exact controllable to the domain. Specifically, this means that {\it for all $0<\varepsilon <1/2$ and for every absolutely continuous scalar, respectively, two-component vector-valued function $y$ with $L_2$-derivative  which has zero values on the intervals $[0,\varepsilon )$ and $(1-\varepsilon ,1]$, any initial function from $L_2(0,1)$, respectively, $L_2(0,1)\times L_2(0,1)$ can be steered to the terminal state $x(T)$ with $x(T)(\theta )=y(\theta )$ on $[\varepsilon ,1-\varepsilon]$}.

We expect this result to hold for the three-fluid parallel-flow heat-exchanger equations as well. We therefore leave this extension to future work, building on the ideas in \cite{MC, BM2005} .

\section{Acknowledgments}
The second author acknowledges support by the Open Access Publication Fund of the University of Duisburg-Essen.


\end{document}